\documentclass{article}
\usepackage{latexsym}
\usepackage{leftidx}
\usepackage{amsmath,amsthm, color, pdfpages}
\usepackage{enumerate}
\usepackage{amssymb}
\usepackage{upgreek}
\usepackage[margin=1.1in,dvips]{geometry}

\def\f12{\frac 1 2}

\def\ga{\gamma}

\def\ep{\epsilon}

\def\om{\omega}

\def\H{\mathcal{H}} 
\def\S{\mathcal{S}} 

\def\Lb{\underline{L}}

\def\pa{\partial}
\def\les{\lesssim}
\def\B{\mathcal{B}}
\def\cL{{\mathcal L}}

\def\f12{\frac 1 2}

\newcommand{\vol}{\textnormal{vol}}

\newcommand{\nabb}{\mbox{$\nabla \mkern-13mu /$\,}}

\newtheorem{Thm}{Theorem}[section]

\newtheorem{Prop}{Proposition}[section]

\newtheorem{Remark}{Remark}[section]
\theoremstyle{definition}

\begin{document}

\title{On the defocusing semilinear wave equations in three space dimension with small power}
\date{}
\author{Dongyi Wei \and Shiwu Yang}
\maketitle
\begin{abstract}
By introducing new weighted vector fields as multipliers, we derive quantitative pointwise estimates for solutions of defocusing semilinear wave equation in $\mathbb{R}^{1+3}$ with pure power nonlinearity for all $1<p\leq 2$. Consequently, the solution vanishes on the future null infinity and decays in time polynomially for all $\sqrt{2}<p\leq 2$. This improves the uniform boundedness result of the second author when $\frac{3}{2}<p\leq 2$.  
\end{abstract}

\section{Introduction}
In this paper, we continue our study on the global pointwise behaviors for solutions to the energy sub-critical defocusing semilinear wave equations
\begin{equation}
  \label{eq:NLW:semi:3d}
  \Box\phi=|\phi|^{p-1}\phi,\quad \phi(0, x)=\phi_0(x),\quad \pa_t\phi(0, x)=\phi_1(x)
\end{equation}
in $\mathbb{R}^{1+3}$ with small power $1<p\leq 2$.

The existence of global solutions in energy space is well known since the work  \cite{velo85:global:sol:NLW} of Ginibre-Velo for the energy sub-critical case $1<p<5$. The long time dynamics of these global solutions concerns mainly two types of questions: The first type is the problem of scattering, namely comparing the nonlinear solutions with linear solutions as time goes to infinity. A natural choice for the linear solution is the associated linear wave. Since linear wave decays in time in space dimension $d\geq 2$, the nonlinear solution approaches to linear wave in certain sense for sufficiently large power $p$. This was shown by Strauss in \cite{Strauss:NLW:decay}  for the super-conformal case $3\leq p<5$ in space dimension three. Extensions could be found for example in \cite{Baez:3DNLW:Groursat}, \cite{Velo87:decay:NLW}, \cite{Hidano:scattering:NLW},  \cite{Pecher82:decay:3d}, \cite{yang:scattering:NLW}.
The latest work \cite{yang:NLW:ptdecay:3D} of the second author shows that the solution scatters to linear wave in energy space for $2.3542<p<5$ in $\mathbb{R}^{1+3}$.
 However the precise asymptotics of the solutions remains unclear for small power $p$. One of the difficulties is that the equation degenerates to linear Klein-Gordon equation when $p$ approaches to the end point $1$.

Another question is to investigate the asymptotic behaviors of the solutions in the pointwise sense, which is plausible in lower dimensions $d\leq 3$. The obstruction in higher dimension is that taking sufficiently many derivatives for energy sub-critical equations is not possible for general $p$. Even for energy critical equations, the global regularity result only holds in space dimension $d\leq 9$ (see for example \cite{Kapitanski94:NLW:n9:cri}). The energy super-critical case remains completely open (see recent breakthrough in \cite{Igor20:blow:NLS} for the blowing up result for the defocusing energy super-critical nonlinear Schr\"{o}dinger equations). The scattering result of Strauss is based on the time decay $t^{\ep-1}$ of the solution, which has been improved in \cite{Bieli:3DNLW} by using conformal compactification method. However this method only works for super-conformal case when $p\geq 3$.
To study the asymptotic behaviors of the solutions with sub-conformal power $p<3$, Pecher in \cite{Pecher82:decay:3d}, \cite{Pecher82:NLW:2d} observed that the potential energy decays in time with a weaker decay rate (comparing to $t^{-2}$ for the super-conformal case). This allows him to obtain polynomial decay in time of the solutions when $p>\frac{1+\sqrt{d^2+4}}{2}$ in space dimension $d=2$ and $3$.

However, Pecher's observation was based on the conformal symmetry of Minkowski spacetime, that is, the time decay of the potential energy is derived by using the conformal Killing vector field as multiplier. As we have seen, the smaller power $p$ leads to the slower decay of the nonlinearity, hence making the analysis more difficult. More precisely, the power $p$ is closely related to the weights in the multipliers. For the super-conformal case $p\geq \frac{d+3}{d-1}$, one can use the conformal Killing vector field with weights $t^2$, which leads to the time decay $t^{-2}$ of the potential energy. For the end point case $p=1$, so far as we know, there is no similar weighted energy estimates for linear Klein-Gordon equations, without appealing to higher order derivatives. This suggests that it is important to use multipliers with proper weights depending on the power $p$ in order to reveal the asymptotic behaviors of the solutions. As the power $p$ varies continuously, it in particular calls for a family of weighted vector fields which are  consistent with the structure of the equations. 
 
The robust new vector field method originally introduced by Dafermos and Rodnianski in \cite{newapp} provides such family of multipliers $r^{\gamma}(\pa_t+\pa_r)$ with $0\leq \ga\leq 2$. Combined with the well known integrated local energy estimates (see for example \cite{mora1}), a pigeon-hole argument leads to the improved time decay of the potential energy. This enables the second author in  \cite{yang:NLW:ptdecay:3D} to  show that the solution decays at least $t^{-\frac{1}{3}}$ for  $2<p<5$ in three space dimension. The lower bound $p>2$  arises due to the fact that the pigeon-hole argument works only for $\gamma>1$. However the multipliers can be used for all $0\leq \ga\leq 2$ and hence uniform boundedness (or spatial decay) of the solution holds for $\frac{3}{2}<p\leq 2$. We see that there is a gap regarding the time decay of the solutions between the cases $p>2$ and $p\leq 2$. Moreover this method fails in lower dimensions $d\leq 2$.

The philosophy that suitable weighted multiplier yields the time decay of the potential energy inspires us to introduce new non-spherically  symmetric weighted vector fields as multipliers in \cite{yang:NLW:1D:p},\cite{yang:NLW:2D} to show the polynomial decay in time of the solution for all $p>1$ in space dimension one and two. This in particular extends the result of Lindblad and Tao in \cite{tao12:1d:NLW}, where an averaged decay of the solution was shown for the defocusing semilinear wave equation in $\mathbb{R}^{1+1}$.
 
The aim of this paper is to investigate the asymptotic behaviors of the solutions in three space dimension with small power $p\leq 2$. Again, the essential idea is to introduce some new weighted vector fields as multipliers, which are partially inspired by our previous work \cite{yang:NLW:2D} in space dimension two. This allows us to obtain potential energy decay for all $1<p<5$ and time decay of the solutions for all $p>\sqrt{2}$, and hence filling the gap left in \cite{yang:NLW:ptdecay:3D}.

To state our main theorem, for some constant $\ga$ and integer $k$, define the weighted energy norm of the initial data
\begin{align*}
  \mathcal{E}_{k,\ga} =\sum\limits_{l\leq k}\int_{\mathbb{R}^3}(1+|x|)^{\ga+2l}(|\nabla^{l+1}\phi_0|^2+|\nabla^l \phi_1|^2)+(1+|x|)^{\ga}|\phi_0|^{p+1}dx.
\end{align*}
We prove in this paper that
\begin{Thm}
\label{thm:main}
Consider the defocusing semilinear wave equation \eqref{eq:NLW:semi:3d} with initial data $(\phi_0, \phi_1)$ such that $\mathcal{E}_{1, 2}$
is finite. Then for all $1<p< 2$,
the solution $\phi$ to the equation \eqref{eq:NLW:semi:3d} exists globally in time and verifies the following asymptotic pointwise estimates
\begin{equation*}
|\phi(t, x)|\leq 
\begin{cases}
C(\sqrt{\mathcal{E}_{1, 2} }+\mathcal{E}_{0, 2}^{\frac{p}{p+1}})(1+|t|+|x|)^{-\frac{p-1}{p+1}}(1+|x|-|t|)^{-\frac{p-1}{p+1}},\quad |x|\geq |t|;\\
C(\sqrt{\mathcal{E}_{1, 2} }+\mathcal{E}_{0, 2}^{\frac{p}{p+1}})(1+|t|+|x|)^{-\frac{(p-1)^2}{p+1}}(1+|t|-|x|)^{\frac{3-2p}{p+1}}
  ,\quad |x|\leq |t|
\end{cases}
\end{equation*}
for some constant $C$ depending only on $p$. For the quadratic nonlinearity with $p=2$, it holds that  
\begin{align*}
  &|\phi(t,x)|\leq C_{\epsilon} (\sqrt{\mathcal{E}_{1, 2} }+\mathcal{E}_{0, 2}^{\frac{2-\ep }{3(1-\ep) }})(1+|t|+|x|)^{-\frac{1 -2\ep}{3}}(1+||x|-|t||)^{-\frac{1-2\ep }{3}}
\end{align*}
for all $0<\ep<\frac{1}{2}$ with constant $C_{\ep}$ depending only on  $\ep$.
 \end{Thm}
We give several remarks.

\begin{Remark} 
The theorem implies that the solution decays along out going null curves ($|t|-|x|$ is constant) for all $1<p\leq 5$ (see \cite{yang:NLW:ptdecay:3D} for the case when $2<p<5$ and \cite{Grillakis:NLW:cri:3d} for the energy critical case). In other words, the solution vanishes on the future (and past) null infinity and  blowing up can only occur at time infinity. It will be of great interest to see whether such blow up can happen particularly for $p$ close to $1$.
\end{Remark}

\begin{Remark} 
\label{remark2}
  In view of the energy conservation, one can easily conclude that the solution grows at most polynomially in time $t$ with rate relying on the power $p$.
  The theorem improves this growth for $1<p\leq \sqrt{2}$ and shows that the solution decays inverse polynomially in time for $\sqrt{2}<p\leq 2$.  In particular, it fills the gap left in \cite{yang:NLW:ptdecay:3D} by the second author, in which only uniform boundedness of the solution was obtained for $\frac{3}{2}<p\leq 2$ while time decay with rate at least $t^{-\frac{1}{3}}$ was shown for $2<p<3$.
\end{Remark}

\begin{Remark} 
 The proof also indicates that the potential energy decays in time for all $1<p\leq 3$, that is,
 \begin{align*}
 \int_{\mathbb{R}^3} |\phi(t, x)|^{p+1}dx\leq C \mathcal{E}_{2, 0}(1+t)^{ 1-p}
 \end{align*}
 with some constant $C$ depending only on $p$. This time decay estimate is stronger than that in \cite{yang:scattering:NLW}.
  
\end{Remark}

Now let's review the main ideas for studying the asymptotic behaviors for defocusing semilinear wave equations. The early pioneering works (for example  \cite{Velo87:decay:NLW}, \cite{Pecher82:decay:3d},\cite{Pecher82:NLW:2d}) relied on the following time decay of the potential energy
\begin{align}
\label{eq:timedecay:3D:Pecher}
  \int_{\mathbb{R}^3}|\phi|^{p+1}dx\leq C (1+t)^{\max\{4-2p , -2\}},\quad 1<p<5,
\end{align}
 obtained by using the conformal Killing vector field $t^2\pa_t+r^2 \pa_r$ ($r=|x|$)   as multiplier.  The new vector field method of  Dafermos and Rodnianski can improve the above time decay in the following way: First the $r$-weighted energy estimates derived by using the vector fields $r^{\ga}(\pa_t+\pa_r)$ with $0\leq \ga\leq 2$ as multipliers show that
  \begin{align*}
    \iint_{\mathbb{R}^{1+3}} r^{\ga-1}|\phi|^{p+1}dxdt\leq C ,\quad 0<\ga<p-1.
  \end{align*}
  However in order to obtain time decay of the potential energy, one then needs to combine the $r$-weighted energy estimate with the integrated local energy estimates. A pigeon-hole argument then leads  to the energy flux decay through the outgoing null hypersurface $\H_u$ (constant $u$ hypersurface with $u=\frac{t-r}{2}$)
  \begin{align*}
    \int_{\H_u} |\phi|^{p+1}d\sigma \leq C(1+|u|)^{-\ga} .
  \end{align*}
Integrating in $u$, we end up with a weighted spacetime bound
\begin{align*}
   \iint_{\mathbb{R}^{1+3}} (1+|u|)^{\ga-1-\ep}|\phi|^{p+1}dxdt\leq C ,\quad \forall 0<\ep<\ga-1
\end{align*}
by assuming $\ga>1$, which requires $p>2$. In view of  the above $r$-weighted energy estimate, one then derives the time weighted spacetime bound
\begin{align*}
  \iint_{\mathbb{R}^{1+3}}(1+t+|x|)^{\ga-1-\ep}|\phi|^{p+1}dxdt\leq C.
\end{align*}
   This improves the above time decay \eqref{eq:timedecay:3D:Pecher} for the sub-conformal case $p<3$ and is sufficient to conclude the time decay estimates of the solutions for $p>2$ in \cite{yang:NLW:ptdecay:3D}.

However, the above new vector field method works only in space dimension $d\geq 3$, due to the lack of integrated local energy estimates in lower dimensions. To improve the asymptotic decay estimates of the solution in space dimension two, we in \cite{yang:NLW:2D} introduced non-spherically symmetric vector fields 
\begin{align*}
X=u_1^{\frac{p-1}{2}}(\partial_t-\partial_1)+u_1^{\frac{p-1}{2}-2}x_2^2(\partial_t+\partial_1)+2u_1^{\frac{p-1}{2}-1}x_2\partial_2,\quad u_1=t-x_1+1
\end{align*}
as multipliers applied to regions bounded by the null 
 hyperplane $\{t=x_1\}$ and the initial hypersurface. The advantage of using such non-spherically symmetric vector fields is that we can make use of the reflection symmetry $x_1\rightarrow -x_1$ as well as rotation symmetries. This enables us to derive the time decay of the potential energy 
\begin{align*}
\int_{\mathbb{R}^2}|\phi|^{p+1}dx\leq C (1+t)^{\max\{-\frac{p-1}{2}, -2\} },\quad\forall p>1
\end{align*}
in space dimension two. This decay rate is consistent with that in higher dimensions. However we emphasize here that this method works for all $p>1$ while   a lower bound $p>1+\frac{2}{d-1}$ was required in higher dimensions (see \cite{yang:scattering:NLW}).

For the three dimensional case when $p\leq 2$, we observe that it is not likely to use multipliers with weights higher than $t$. Although the vector fields $r^{\gamma}(\pa_t+\pa_r)$ can be used for all $0\leq \gamma\leq 2$, it does not contain weights in time. In particular these vector fields can only lead to the spatial decay of the solution instead of time decay. Inspired by our previous work in space dimension two, we introduce new non-spherical symmetric weighted vector fields
 \begin{align*}
X=u^2f(u)(\partial_t-\partial_1)+& f(u)(x_2^2+x_3^2)(\partial_t+\partial_1)+2u f(u)(x_2\partial_2+x_3\partial_3)+(2f(u)+1)\partial_t,\\
 &  u=t-x_1,\quad  f(u)=(1+|\max\{u, 0\}|^2)^{\frac{p-3}{2}}.
\end{align*}
 as multipliers. By applying this vector field to the region bounded by the null hyperplane $\{t=x_1-1\}$, the initial hypersurface and the constant $t$-hypersurface, we can derive that 
 \begin{align*}
 \int_{x_1\leq 0} (1+t)^{p-1} |\phi(t, x)|^{p+1}dx\leq C \mathcal{E}_{2, 0}.
 \end{align*}
 Using the symmetry $x_1\rightarrow -x_1$, we then conclude the time decay of the potential energy 
  \begin{align*}
 \int_{\mathbb{R}^3}  |\phi(t, x)|^{p+1}dx\leq C \mathcal{E}_{2, 0} (1+t)^{1-p}
 \end{align*}
 for all $1<p\leq 2$. In view of this, we believe that such method can also lead to the time decay of potential energy in higher dimensions for the full energy sub-critical case.

Regarding the pointwise decay estimates for the solution, we rely on the representation formula. To control the nonlinearity, we apply the above vector fields to the region bounded by the backward light cone $\mathcal{N}^{-}(q)$ emanating from the point $q\in\mathbb{R}^{1+3}$.  To simplify the analysis, we can assume that  $q=(t_0,x_0)$, $x_0=(r_0=|x_0|,0,0)$. This gives the weighted energy estimate
 \begin{align*}
 \int_{\mathcal{N}^{-}(q)}(|t_0-r_0|^{p-1}(1-\frac{x_1-r_0}{|x-x_0|})+|t_0+r_0|^{p-1}(1+\frac{x_1-r_0}{|x-x_0|})+1)|\phi|^{p+1} dx
\leq C \mathcal{E}_{0, 2},
\end{align*}
which is sufficient to conclude the pointwise estimate for the solution in the interior region $|x|\leq t$. The better decay estimates in the exterior region are based on the weighted energy estimate 
 \begin{align*}
 \int_{\mathcal{N}^{-}(q)\cap\{x_1\geq t\}}|r_0+t_0\frac{x_1-r_0}{|x-x_0|}||\phi|^{p+1} dx
\leq C \mathcal{E}_{0, 1}.
\end{align*} 
obtained by using the Lorentz rotation vector field $x_1\pa_t+t\pa_1$ as multiplier.

\textbf{Acknowledgments.} S. Yang is partially supported by NSFC-11701017.

\section{Preliminaries and notations}
\label{sec:notation}

Additional to the Cartesian coordinates $(t, x)=(t, x_1, x_2, x_3)$ for the Minkowski spacetime $\mathbb{R}^{1+3}$, we will also use the null frame 
\[
L_1=\pa_t+\pa_1,\quad \underline{L}_1=\pa_t-\pa_1
\]
with $\pa_i$ be shorthand for $\pa_{x_i}$. For fixed point $q=(t_0, x_0)\in\mathbb{R}^{1+3}$,   
let $(\tilde{t}, \tilde{x})$ be the new Cartesian coordinates centered at $q$. More precisely, define
\[
\tilde{t}=t-t_0,\quad \tilde{x}=x-x_0,\quad \tilde{r}=|\tilde{x}|,\quad \tilde{\om}=\frac{\tilde{x}}{|\tilde{x}|},\quad \tilde{L}=\pa_{\tilde{t}}+\pa_{\tilde{r}},\quad \tilde{\Lb}=\pa_{\tilde{t}}-\pa_{\tilde{r}}.
\]
By translation invariance, note that 
\begin{align*}
\pa_{\tilde{t}}=\pa_{t},\quad \pa_{\tilde{r}}=\tilde{\om}\cdot \tilde{\nabla}=\tilde{\om}\cdot \nabla.
\end{align*}
Here $\nabla$ is the spatial gradient while $\tilde{\nabla}$ is the associated one centered at $q$.

For vector fields $X, Y$ in $\mathbb{R}^{1+3}$, we use the geometric notation $ \langle X, Y\rangle$ meaning the inner product of these two vector fields under the flat Minkowski metric $m_{\mu\nu}$ with non-vanishing components 
\begin{align*}
m_{00}=-1, \quad m_{ii}=1,\quad i=1, 2, 3. 
\end{align*}
Raising and lowering indices are carried out with respect to this metric in the sequel.

As the wave equation is time reversible, without loss of generality, we only consider the case in the future $t\geq 0$. For $q=(t_0, x_0)\in \mathbb{R}^{1+3}$,
let $\mathcal{J}^{-}(q)$ be the causal past 
   \begin{align*}
   \mathcal{J}^{-}(q):=\{(t, x)| |x-x_0|\leq t_0-t,\quad t\geq 0\}.
 \end{align*}
 The boundary contains the past null cone $\mathcal{N}^{-}(q)$ emanating from $q$, that is,
 \begin{align*}
   \mathcal{N}^{-}(q):=\{(t, x)| t_0-t=|x-x_0|,\quad t\geq 0\}.
 \end{align*}
 For $r>0$, we use
$\B_q(r)$ to denote the spatial ball centered at $q=(t_0, x_0)$ with radius $r$. More precisely
\begin{align*}
  \B_q( r)=\{(t,x)| t=t_0, |x-x_0|\leq r\}.
\end{align*}
 The boundary of $\B_q( r)$ is the 2-sphere $\S_q(r)$.

Finally to avoid too many constants, we make a convention that $A\les B$ means there exists a constant $C$, depending only on $p$ and the small constant $ 0<\epsilon<\frac{1}{2}$ such that $A\leq CB$.

\section{Weighted energy estimates through backward light cones}
Following the framework established early in \cite{Strauss:NLW:decay} and developed in \cite{Pecher82:decay:3d},  \cite{yang:NLW:ptdecay:3D}, potential energy decay is of crucial importance to deduce the asymptotic long time behavior for the solution. We begin with the following weighted potential energy estimate through backward light cones. 
\begin{Prop}
\label{prop:EF}
Assume that $1<p\leq 2$.
Let $q=(t_0,r_0, 0, 0)$ be a point in $\mathbb{R}^{1+3}$ with $t_0, r_0 \geq0$. Then the solution $\phi$ of the nonlinear wave equation \eqref{eq:NLW:semi:3d} verifies the following weighted energy estimates
\begin{align}
\label{eq:Ef}
&\int_{\mathcal{N}^{-}(q)}(|t_0-r_0|^{p-1}(1-\tilde{\om}_1)+|t_0+r_0|^{p-1}(1+\tilde{\om}_1)+1)|\phi|^{p+1} d\sigma
\leq C \mathcal{E}_{0, 2},
\\
\label{eq:Ef1}
&\int_{\mathcal{N}^{-}(q)\cap\{x_1\geq t\}}|r_0+t_0\tilde{\om}_1||\phi|^{p+1} d\sigma 
\leq C \mathcal{E}_{0, 1} 
\end{align}
for some constant $C$ depending only on $p$. Here $d\sigma$ is the surface measure and $\tilde{\om}_1=\frac{x_1-r_0}{\sqrt{(x_1-r)^2+x_2^2+x_3^2}}.$
\end{Prop}
\begin{proof}
Recall the energy momentum tensor for the scalar field $\phi$
\begin{align*}
  T[\phi]_{\mu\nu}=\pa_{\mu}\phi\pa_{\nu}\phi-\f12 m_{\mu\nu}(\pa^\ga \phi \pa_\ga\phi+\frac{2}{p+1} |\phi|^{p+1}).
\end{align*}
For any vector fields $X$, $Y$ and any function $\chi$, define the current
\begin{equation*}
J^{X, Y, \chi}_\mu[\phi]=T[\phi]_{\mu\nu}X^\nu -
\f12\pa_{\mu}\chi \cdot|\phi|^2 + \f12 \chi\pa_{\mu}|\phi|^2+Y_\mu.
\end{equation*}
Then for solution $\phi$ of equation \eqref{eq:NLW:semi:3d} and any domain $\mathcal{D}$ in $\mathbb{R}^{1+3}$, we have the energy identity
\begin{equation}
\label{eq:energy:id}
\iint_{\mathcal{D}}\pa^\mu  J^{X,Y,\chi}_\mu[\phi] d\vol =\iint_{\mathcal{D}}div(Y)+ T[\phi]^{\mu\nu}\pi^X_{\mu\nu}+
\chi \pa_\mu\phi\pa^\mu\phi -\f12\Box\chi\cdot|\phi|^2 +\chi |\phi|^{p+1}  d\vol.
\end{equation}
Here $\pi^X=\f12 \cL_X m$  is the deformation tensor for the vector field $X$.

\def\L{{L_1}}
\def\Lu{{\underline{L}_1}}
\def\u{u}
\def\g{g}
\def\Z{Z}
\def\Za{\widetilde{Z}}
\def\Zb{\widehat{Z}}

In the above energy identity, choose the vector fields $X,$ $Y$ and the function $\chi$ as follows:
\begin{align*}
&X=\u^2f(\u) \Lu 
+f(\u)(x_2^2+x_3^2) \L 
+2\u f(\u)(x_2\partial_2+x_3\partial_3)+(2f(\u)+1)\partial_t, \\
 &Y=(\L\g) S-(S\g+3\g)\L,\quad \chi=2\u f(\u),\quad \g=f(\u)|\phi|^2,\quad \u=t-x_1,\\
 &f(\u)=1\quad\text{for}\quad \u\leq0;\quad f(\u)=(1+\u^2)^{\frac{p-3}{2}}\quad\text{for}\quad \u\geq0.
\end{align*}
Here $S=t\partial_t+r\pa_r$ is the scaling vector field and $ \L=\partial_t+\partial_1$, $\Lu=\partial_t-\partial_1$. We then compute that
\begin{align*}
&\nabla_{\L}X= 0,\quad \nabla_{j}X=2f(\u)x_j\L+2\u f(\u)\partial_j,\quad  j=2,3,\\ 
&\nabla_{\Lu}X= 2(2\u f(\u)+\u^2f'(\u))\Lu+2f'(\u)(x_2^2+x_3^2)\L+4(f(\u)+\u f'(\u))(x_2\partial_2+x_3\partial_3)+4f'(\u)\partial_t.
\end{align*}
In particular the non-vanishing components of the deformation tensor $\pi_{\mu\nu}^X$ are
\begin{align*}
\pi^X_{\L\Lu}&=-2(2\u f(\u)+(\u^2+1)f'(\u)),\quad \pi^X_{\Lu\Lu}=-4f'(\u)(x_2^2+x_3^2+1),\\
 \pi^X_{\Lu \pa_j}&=2\u f'(\u)x_j,\quad \pi^X_{\pa_j\pa_j}=2\u f(\u),\quad j=2,3.
\end{align*}
We also have $div(\L)=0$, $div(S)=4$, $[\L,S]=\L$ and
\begin{align*}
div(Y)&=S\L\g+4\L\g-\L(S\g+3\g)=-[\L,S]\g+\L\g=0.
\end{align*}
Since $\Box\chi=0 $, we therefore can compute that
\begin{align*}
& T[\phi]^{\mu\nu}\pi^X_{\mu\nu}+
\chi \pa_\mu\phi \pa^\mu\phi  +\chi |\phi|^{p+1}-\f12 \Box\chi |\phi|^2\\
=&-(2\u f(\u)+(\u^2+1)f'(\u))(|(\partial_2,\partial_3)\phi|^2+\frac{2}{p+1}|\phi|^{p+1})\\&+ 2\u f(\u)(|(\partial_2,\partial_3)\phi|^2-\pa^\mu \phi \pa_\mu\phi-\frac{2}{p+1}|\phi|^{p+1})\\
& -f'(\u)(x_2^2+x_3^2+1) |\L\phi|^2-2\u f'(\u)(x_2\partial_2\phi+x_3\partial_3\phi) \L\phi+
2\u f(\u)( \pa_\mu\phi \pa^\mu\phi  +|\phi|^{p+1}) \\
=&-f'(\u)(|(x_2,x_3)\L\phi+\u(\partial_2,\partial_3)\phi|^2+|(\partial_2,\partial_3,\L)\phi|^2)+\frac{2|\phi|^{p+1}}{p+1} ((p-3)\u f(\u)-(\u^2+1)f'(\u)).
\end{align*}
Since we are restricting to the range $1<p\leq 2<3$, in view of the definition of the function $f$, we note that 
when $\u\leq0$,
 $$f'(\u)=0, (p-3)\u f(\u)=(p-3)\u\geq0,$$ and 
 $$f'(\u)=(p-3)\u(1+\u^2)^{\frac{p-3}{2}}\leq0, \quad (p-3)\u f(\u)=(\u^2+1)f'(\u), \quad \u\geq0.$$
In particular we always have
\begin{align}\label{eq:f1}
f'(\u)\leq0,\quad (p-3)\u f(\u)\geq(\u^2+1)f'(\u),
\end{align}
which implies that the bulk integral is nonnegative.

Now take the domain $\mathcal{D}$ to be $\mathcal{J}^{-}(q)$ with boundary $\B_{(0, x_0)}(t_0)\cup \mathcal{N}^{-}(q)$. In view of Stokes' formula, the left hand side of the energy identity \eqref{eq:energy:id} is reduced to integrals on the initial hypersurface $\B_{(0, x_0)}(t_0)$ and on the backward light cone $\mathcal{N}^{-}(q)$. Since the bulk integral on the right hand side is nonnegative, we conclude that
\begin{align}\label{eq:N-q}
  \int_{\mathcal{N}^{-}(q)}i_{J^{X, Y, \chi}[\phi]}d\vol+\int_{\B_{(0, x_0)}(t_0)}i_{J^{X, Y, \chi}[\phi]}d\vol\geq 0.
\end{align}
For the integral on the initial hypersurface $\B_{(0, x_0)}(t_0)$, we compute it under the coordinates $(t, x)$
\begin{align*}
  i_{J^{X,Y,\chi}[\phi]}d\vol=&-(J^{X,Y, \chi}[\phi])^{0}dx= ( T[\phi]_{0 \nu}X^\nu -
\f12 \pa_t\chi |\phi|^2 + \f12 \chi\cdot \pa_t |\phi|^2+Y_0)  dx\\
=&\f12f(\u)\left( (x_2^2+x_3^2+\u^2) (|\partial_t\phi|^2+|\nabla\phi|^2+\frac{2}{p+1}|\phi|^{p+1}) +2\u   \pa_t |\phi|^2\right. \\
& \left.+2 (x_2^2+x_3^2-\u^2)\pa_t\phi\cdot \partial_1\phi+4\u  \partial_t\phi (x_2\partial_2\phi+x_3\partial_3\phi) \right)  dx-(f(\u)+\u f'(\u))|\phi|^2  dx\\
&+\f12(2f(\u)+1)(|\partial_t\phi|^2+|\nabla\phi|^2+\frac{2}{p+1}|\phi|^{p+1})dx+(S\g+3\g-t\L\g)dx\\
=&\f12 f(\u)\left( |(x_2,x_3)(\partial_t\phi+\partial_1\phi)+\u(\partial_2,\partial_3)\phi|^2+|\u(\partial_t\phi-\partial_1\phi)+(x_2\partial_2+x_3\partial_3)\phi+2\phi|^2\right.\\
& \qquad \quad \left.+(x_2\pa_3\phi-x_3\pa_2\phi)^2+\frac{2(x_2^2+x_3^2+\u^2)}{p+1}|\phi|^{p+1} \right)dx\\
&+\f12(2f(\u)+1)(|\partial_t\phi|^2+|\nabla\phi|^2+\frac{2}{p+1}|\phi|^{p+1})dx.
\end{align*}
Here we used the relation $\L \u=0$, $S \u=\u$ and the following computation
\begin{align*}
S\g+3\g-t\L\g&=S(f(\u)|\phi|^2)+3f(\u)|\phi|^2-t\L(f(\u)|\phi|^2)\\&=f(\u)S|\phi|^2+(\u f'(\u)+3f(\u))|\phi|^2-tf(\u)\L|\phi|^2
\\&=f(\u)(-\u\partial_1+x_2\partial_2+x_3\partial_3)|\phi|^2+(\u f'(\u)+3f(\u))|\phi|^2.
\end{align*}
Since $0<f(\u)\leq 1$ and $\u=-x_1$ on $\B_{(0, x_0)}(t_0)$, by using Hardy's inequality to control $|\phi|^2$, we can bound the integral on the initial hypersurface by the initial weighted energy
\begin{align}
\label{eq:bxt0}
 \int_{\B_{(0, x_0)}(t_0)} i_{J^{X, Y,\chi}[\phi]}d\vol \leq C \mathcal{E}_{0, 2}.
\end{align}
Next we compute the boundary integral on the backward light cone $\mathcal{N}^{-}(q)$. The surface measure is of the form
\begin{align*}
-i_{J^{X, Y,\chi}[\phi]}d\vol=J_{\tilde{\Lb}}^{X, Y,\chi}[\phi]d\sigma = ( T[\phi]_{\tilde{\Lb}\nu}X^\nu -
\f12(\tilde{\Lb}\chi) |\phi|^2 + \f12 \chi\cdot\tilde{\Lb}|\phi|^2 +Y_{\tilde{\Lb}}) d\sigma.
\end{align*}
Here we recall that the null frame $\{\tilde{\Lb }, \tilde{L}, \tilde{e}_1, \tilde{e}_2\}$ is centered at the point $q$. Since
$$\L \u=0, \quad S \u=\u=-\langle S,  \L\rangle, $$ we have
\begin{align*}
\L\g &=\L(f(\u)|\phi|^2)=f(\u)\L|\phi|^2,\\
 S\g+3\g&=S(f(\u)|\phi|^2)+3f(\u)|\phi|^2=f(\u)S|\phi|^2+(\u f'(\u)+3f(\u))|\phi|^2.
\end{align*}
Note that  $$\tilde{\Lb} \u=\tilde{\Lb}( t)-\tilde{\Lb} (x_1)=-\langle\tilde{\Lb},\partial_t\rangle-\langle\tilde{\Lb},\partial_1\rangle=-\langle\tilde{\Lb},\L\rangle. $$
We therefore can compute that 
\begin{align*}
&-\f12(\tilde{\Lb}\chi) |\phi|^2 + \f12 \chi\cdot\tilde{\Lb}|\phi|^2 +Y_{\tilde{\Lb}}\\
&=-\tilde{\Lb}(\u f(\u)) |\phi|^2 +\u f(\u)\tilde{\Lb}|\phi|^2+(\L\g)\langle S,\tilde{\Lb}\rangle-(S\g+3\g)\langle\L,\tilde{\Lb}\rangle\\
&=-(\tilde{\Lb}\u)(f(\u)+\u f'(\u)) |\phi|^2 +\u f(\u)\tilde{\Lb}|\phi|^2+  f(\u)(\L|\phi|^2)\langle S,\tilde{\Lb}\rangle\\
&\quad +(f(\u)S|\phi|^2+(\u f'(\u)+3f(\u))|\phi|^2)\tilde{\Lb} \u\\
&= f(\u)(\L|\phi|^2)\langle S,\tilde{\Lb}\rangle+(f(\u)S|\phi|^2+2f(\u)|\phi|^2)\tilde{\Lb} \u +\u f(\u)\tilde{\Lb}|\phi|^2\\
&=f(\u)\Z|\phi|^2-f(\u)(S|\phi|^2+2|\phi|^2)\langle\tilde{\Lb},\L\rangle\\
&=f(\u) \left(\Za|\phi|^2-2|\phi|^2\langle\tilde{\Lb},\L\rangle\right).
\end{align*}
Here the vector fields $Z$, $\Za$ are given by 
\begin{align*}
&\Z=\langle S,\tilde{\Lb}\rangle\L+\u\tilde{\Lb}=\langle S,\tilde{\Lb}\rangle\L-\langle S,\L\rangle\tilde{\Lb},\quad \Za=\Z-\langle\tilde{\Lb},\L\rangle S.
\end{align*}
Now we write the vector field $X$ as
\begin{align*}
&X=f(\u)X_0+(2f(\u)+1)\partial_t,\quad X_0=\u^2(\partial_t-\partial_1)+(x_2^2+x_3^2)(\partial_t+\partial_1)+2\u (x_2\partial_2+x_3\partial_3).
\end{align*}
The vector field $X_0$ can be further written as  
\begin{align*}
&X_0=(r^2-t^2)(\partial_t+\partial_1)+2\u S=\langle S,S\rangle\L-2\langle S,\L\rangle S.
\end{align*}
Now we expand the current
\begin{align*}
J_{\tilde{\Lb}}^{X, Y,\chi}[\phi] &=  T[\phi]_{\tilde{\Lb}\nu}X^\nu -
\f12(\tilde{\Lb}\chi) |\phi|^2 + \f12 \chi\cdot\tilde{\Lb}|\phi|^2 +Y_{\tilde{\Lb}}\\
&=f(u)T[\phi]_{\tilde{\Lb}X_0}+(2f(\u)+1)T[\phi]_{\tilde{\Lb}\partial_t}+f(\u)(\Za|\phi|^2-2|\phi|^2\langle\tilde{\Lb},\L\rangle)
\\
&=f(u)(T[\phi]_{\tilde{\Lb}X_0}+\Za|\phi|^2-2|\phi|^2\langle\tilde{\Lb},\L\rangle)+(2f(\u)+1)T[\phi]_{\tilde{\Lb}\partial_t}.
\end{align*}
For the second term, note that 
\begin{align*}
&T[\phi]_{\tilde{\Lb}\partial_t}=\f12 |{\tilde{\Lb}}\phi|^2 + \f12 (|\tilde{\nabb}\phi|^2+\frac{2}{p+1}|\phi|^{p+1})\geq\frac{1}{p+1}|\phi|^{p+1}.
\end{align*}
For the first term we claim that
\begin{align}
\label{eq:TX0}
&T[\phi]_{\tilde{\Lb}X_0}+\Za|\phi|^2-2|\phi|^2\langle\tilde{\Lb},\L\rangle\geq-\langle\tilde{\Lb},X_0\rangle\frac{|\phi|^{p+1}}{p+1}.
\end{align}
At any fixed point of the backward light cone $\mathcal{N}^{-}(q)$, we prove the above claim by discussing three different cases:
\begin{itemize}
\item[(i)] If the vector field $X_0$ vanishes, that is $X_0=0$, then 
\[
0=u=t-x_1=x_2=x_3.
\]
In particular, we have 
 $$S=t\L, \quad -\langle\tilde{\Lb},\L\rangle=\langle \pa_{\tilde{r}}-\pa_{\tilde{t}}, \pa_t+\pa_1 \rangle =\langle \tilde{r}^{-1} \pa_{1}-\pa_{t}, \pa_t+\pa_1 \rangle =1+\widetilde{\omega}_1\geq 0 .$$
 Here recall that $q=(t_0, r_0, 0, 0)$ and $x_2=x_3=0$ for this case. This implies that 
\begin{align*}
&\Za=\langle S,\tilde{\Lb}\rangle\L-\langle\tilde{\Lb},\L\rangle S=0,\\
&T[\phi]_{\tilde{\Lb}X_0}+\Za|\phi|^2-2|\phi|^2\langle\tilde{\Lb},\L\rangle=-2|\phi|^2\langle\tilde{\Lb},\L\rangle\geq 0.
\end{align*}
Hence the above claim holds. 
\item[(ii)] If $X_0\neq0$ and the vector fields $X_0$, $ \tilde{\Lb}$ are linearly dependent, 
by comparing the coefficients of $\pa_t=\pa_{\tilde{t}}$, we conclude that $X_0=\lambda\tilde{\Lb}$ with  $\lambda=\u^2+x_2^2+x_3^2>0$. Recall the definition for the vector fields $X_0$, $\Za$. We can show that 
\begin{align*}
\lambda\Za&=\langle S,X_0\rangle\L-\langle S,\L\rangle X_0-\langle X_0,\L\rangle S\\
&=-\langle S,S\rangle\langle S,\L\rangle\L-\langle S,\L\rangle X_0+2\langle S,\L\rangle^2 S\\
&=-2\langle S,\L\rangle X_0=2\u \lambda\tilde{\Lb},
\end{align*}
which in particular implies that $\Za=2u\tilde{\Lb}$. Note that 
\begin{align*}
& -\lambda\langle\tilde{\Lb},\L\rangle=-\langle X_0,\L\rangle=2\langle S,\L\rangle^2=2\u^2.
\end{align*}
We then can demonstrate that 
\begin{align*}
&T[\phi]_{\tilde{\Lb}X_0}+\Za|\phi|^2-2|\phi|^2\langle\tilde{\Lb},\L\rangle=\lambda|\tilde{\Lb}\phi|^2+2\u\tilde{\Lb}|\phi|^2
+4\u^2|\phi|^2/\lambda=|\lambda\tilde{\Lb}\phi+2\u\phi|^2 \lambda^{-1}\geq 0.
\end{align*}
The above claim follows as $\langle\tilde{\Lb},X_0\rangle=0$.

\item[(iii)] The remaining case is when $X_0\neq 0$ and the vector fields $X_0$, $ \tilde{\Lb}$ are linearly independent. We write
\begin{align*}
&X_0=(\u^2+x_2^2+x_3^2)(\partial_t+\hat{\omega}\cdot\nabla),\quad \tilde{\Lb}=\partial_t-\widetilde{\omega}\cdot\nabla,\\
&\hat{\omega}=(x_2^2+x_3^2-\u^2,2\u x_2,2\u x_3)/(\u^2+x_2^2+x_3^2),\quad |\hat{\omega}|=1,\quad \hat{\omega}\neq -\widetilde{\omega},\\
&\langle\tilde{\Lb},\tilde{\Lb}\rangle=\langle X_0,X_0\rangle=0,\quad -\langle\tilde{\Lb},X_0\rangle=(\u^2+x_2^2+x_3^2)(1+\hat{\omega}\cdot\widetilde{\omega})>0.
\end{align*} 
Here we may note that $\nabla=\tilde{\nabla}$. In particular we see that $ \tilde{\Lb}$, $X_0$ are null vectors which are linearly independent. We thus can construct a null frame $\{X_0,\tilde{\Lb},\hat{e}_1,\hat{e}_2\}$ such that $\langle\tilde{\Lb},\hat{e}_j\rangle=\langle X_0,\hat{e}_j\rangle=0$, $\langle\hat{e}_j,\hat{e}_j\rangle=1$, $\langle\hat{e}_1,\hat{e}_2\rangle=0$ for $j=1,2$. Notice that
\begin{align*}
\langle\Z,S\rangle &=\langle S,\tilde{\Lb}\rangle\langle\L,S\rangle-\langle S,\L\rangle\langle\tilde{\Lb},S\rangle=0,\quad \langle\Z,\tilde{\Lb}\rangle=\langle S,\tilde{\Lb}\rangle\langle\L,\tilde{\Lb}\rangle,\\
\langle\Z,\L\rangle &=-\langle S,\L\rangle\langle\tilde{\Lb},\L\rangle,\quad \langle\Za,\tilde{\Lb}\rangle=\langle\Z,\tilde{\Lb}\rangle-\langle\tilde{\Lb},\L\rangle \langle S,\tilde{\Lb}\rangle=0,\\
\langle\Za,X_0\rangle&=\langle\Z-\langle\tilde{\Lb},\L\rangle S,\langle S,S\rangle\L-2\langle S,\L\rangle S\rangle\\
&=\langle\Z,\L\rangle \langle S,S\rangle+\langle\tilde{\Lb},\L\rangle \langle S,\L\rangle\langle S,S\rangle=0.
\end{align*}
The above computation in particular shows that $\Za\in\text{span}\{\hat{e}_1,\hat{e}_2\} $. We hence can write that 
$$\Za=a_1\hat{e}_1+a_2\hat{e}_2. $$
On the other hand,  we also have
\begin{align*}
&\langle\tilde{\Lb},X_0\rangle=\langle\tilde{\Lb},\langle S,S\rangle\L-2\langle S,\L\rangle S\rangle=\langle\tilde{\Lb},\L\rangle\langle S,S\rangle-2\langle\tilde{\Lb},S\rangle\langle S,\L\rangle ,\\
&\langle\Za,\Za\rangle=\langle\Z,\Z\rangle+\langle\tilde{\Lb},\L\rangle^2\langle S,S\rangle=-2\langle S,\tilde{\Lb}\rangle\langle\L,\tilde{\Lb}\rangle\langle S,\L\rangle+\langle\tilde{\Lb},\L\rangle^2\langle S,S\rangle =\langle\tilde{\Lb},X_0\rangle\langle\L,\tilde{\Lb}\rangle.
\end{align*}
Let $$p_1=-\langle\tilde{\Lb},X_0\rangle, \quad p_2=-\langle\L,\tilde{\Lb}\rangle=1+\tilde{\om}_1.$$ 
Then the above computations show that 
\[
p_1>0,\quad p_2\geq 0, \quad p_1p_2=a_1^2+a_2^2.
\]
We therefore can compute that 
\begin{align*}
&T[\phi]_{\tilde{\Lb}X_0}+\Za|\phi|^2-2|\phi|^2\langle\tilde{\Lb},\L\rangle+\langle\tilde{\Lb},X_0\rangle\frac{|\phi|^{p+1}}{p+1}\\
=&\tilde{\Lb}\phi X_0\phi-\frac{1}{2}\langle \tilde{\Lb}, X_0 \rangle (\pa^\ga\phi\pa_\ga\phi +\frac{2}{p+1}|\phi|^{p+1})+\Za|\phi|^2-2|\phi|^2\langle\tilde{\Lb},\L\rangle+\langle\tilde{\Lb},X_0\rangle\frac{|\phi|^{p+1}}{p+1}\\
=&\tilde{\Lb}\phi X_0\phi-\frac{1}{2}\langle \tilde{\Lb}, X_0 \rangle (\pa^{X_0}\phi X_0\phi +\pa^{\tilde{\Lb}}\phi \tilde{\Lb}\phi )+\f12 p_1(|\hat{e}_1\phi|^2+|\hat{e}_2\phi|^2)+2(a_1\hat{e}_1\phi+a_2\hat{e}_2\phi)\phi+2p_2|\phi|^2\\
=& \f12 p_1 (|\hat{e}_1\phi+2p_1^{-1} a_1\phi|^2+|\hat{e}_2\phi+2 p_1^{-1} a_2 \phi|^2)\geq0.
\end{align*}
This means that the above claim \eqref{eq:TX0} always holds. 
\end{itemize}
In view of the estimate \eqref{eq:TX0}, we then conclude that
\begin{align*}
&J_{\tilde{\Lb}}^{X, Y,\chi}[\phi]\geq\frac{-f(\u)\langle\tilde{\Lb},X_0\rangle +2f(\u)+1}{p+1}|\phi|^{p+1}.
\end{align*}
Now we compute $-\langle\tilde{\Lb},X_0\rangle$ under the coordinates $(\tilde{t}, \tilde{x})$ centered at $q=(t_0, r_0,0,0)$. We have
\begin{align*}
&(t,x)=(t_0+\widetilde{t},r_0+\widetilde{x}_1,\widetilde{x}_2,\widetilde{x}_3)
=(t_0+\widetilde{t},r_0+\widetilde{r}\widetilde{\omega}_1,\widetilde{r}\widetilde{\omega}_2,\widetilde{r}\widetilde{\omega}_3).
\end{align*}
Note that on the backward light cone $\mathcal{N}^{-}(q)$, we also have $ \widetilde{t}=-\widetilde{r}$. We thus 
can compute that
\begin{align*}
-\langle\tilde{\Lb},X_0\rangle&=(\u^2+x_2^2+x_3^2)+(x_2^2+x_3^2-\u^2)\widetilde{\omega}_1+2\u x_2\widetilde{\omega}_2+2\u x_3\widetilde{\omega}_3\\
&=\u^2(1-\widetilde{\omega}_1)+(x_2^2+x_3^2)(1+\widetilde{\omega}_1)+2\u\widetilde{r}(\widetilde{\omega}_2^2+\widetilde{\omega}_3^2)
\\
&=\u^2(1-\widetilde{\omega}_1)+\widetilde{r}^2(1-\widetilde{\omega}_1^2)(1+\widetilde{\omega}_1)+2\u\widetilde{r}(1-\widetilde{\omega}_1^2)
\\
&=(\u+\widetilde{r}(1+\widetilde{\omega}_1))^2(1-\widetilde{\omega}_1)\\
&=(t_0+\tilde{t}-r_0-\widetilde{x}_1+\widetilde{r}+\widetilde{x}_1))^2(1-\widetilde{\omega}_1)\\
&=(t_0-r_0)^2(1-\widetilde{\omega}_1)\geq0
\end{align*}
on $\mathcal{N}^{-}(q)$. For the case when $u\leq 0$, by the definition of $f(u)$, we have the lower bound 
\begin{align*}
-f(\u)\langle\tilde{\Lb},X_0\rangle+2f(\u)+1=3+(t_0-r_0)^2(1-\widetilde{\om}_1)\geq 1+ |t_0-r_0|^{p-1} (1-\widetilde{\om}_1)
\end{align*}
as $1<p\leq 2$ and $|\widetilde{\om}_1|\leq 1$. For the case when $u>0$, note that on $\mathcal{N}^{-}(q)$
\begin{align*}
0<u=t-x_1=t_0+\tilde{t}-r_0-\tilde{r}\widetilde{\om}_1=t_0-r_0-\tilde{r}(1+\widetilde{\om}_1)\leq t_0-r_0.
\end{align*}
We therefore can bound that 
\begin{align*}
-f(\u)\langle\tilde{\Lb},X_0\rangle+2f(\u)+1&= 1+2(1+u^2)^{\frac{p-3}{2}}+(t_0-r_0)^2(1-\widetilde{\om}_1)(1+u^2)^{\frac{p-3}{2}}\\
&\geq 1+ (1+(t_0-r_0)^{2} )(1-\widetilde{\om}_1)(1+u^2)^{\frac{p-3}{2}}\\
&\geq 1+ |t_0-r_0|^{p-1} (1-\widetilde{\om}_1).
\end{align*}
Here again we used the assumption that $p<3$. Hence in any case, we have shown that 
\begin{align*}
&J_{\tilde{\Lb}}^{X, Y,\chi}[\phi]\geq\frac{-f(\u)\langle\tilde{\Lb},X_0\rangle+2f(\u)+1}{p+1}|\phi|^{p+1}
\geq\frac{|t_0-r_0|^{p-1} (1-\widetilde{\omega}_1)+1}{p+1}|\phi|^{p+1}.
\end{align*}
In other words, we have the lower bound for the integral on the backward light cone $\mathcal{N}^{-}(q)$
\begin{align*}
  -\int_{\mathcal{N}^{-}(q)}i_{J^{X, Y, \chi}[\phi]}d\vol=\int_{\mathcal{N}^{-}(q)}J_{\tilde{\Lb}}^{X, Y,\chi}[\phi]d\sigma \geq\int_{\mathcal{N}^{-}(q)}\frac{|t_0-r_0|^{p-1} (1-\widetilde{\omega}_1)+1}{p+1}|\phi|^{p+1}d\sigma,
\end{align*}
which together with estimates \eqref{eq:N-q}, \eqref{eq:bxt0} implies that 
\begin{align}
\label{eq:000}
  \int_{\mathcal{N}^{-}(q)}\frac{|t_0-r_0|^{p-1} (1-\widetilde{\omega}_1)+1}{p+1}|\phi|^{p+1}dx\leq C \mathcal{E}_{0, 2}.
\end{align}
To conclude estimate \eqref{eq:Ef} of the proposition, we make use of the reflection symmetry additional to the spherical symmetry. More precisely, by changing variable $x_1\rightarrow -x_1$ in the above argument, that is, setting $u=t+x_1$ and $L_1=\pa_t-\pa_1$, $\underline{L}_1=\pa_t+\pa_1$ accordingly (the point $q$ is still fixed), we also have
\begin{align}
\label{eq:111}
   \int_{\mathcal{N}^{-}(q)}\frac{|t_0+r_0|^{p-1} (1+\widetilde{\omega}_1)+1}{p+1}|\phi|^{p+1}dx\leq C \mathcal{E}_{0, 2}.
\end{align}
Alternative interpretation is that the above estimate \eqref{eq:000} also holds at the point $q^{-}=(t_0, -r_0, 0, 0)$ $r_0<0$ (with positive sign of $\widetilde{\omega}_1$). Then by spherical symmetry, the associated estimate is valid at point $q=(t_0, r_0, 0, 0)$, which is exactly the estimate \eqref{eq:111}. These two estimates lead to \eqref{eq:Ef}.

\bigskip

To finish the proof for the Proposition, it remains to show estimate \eqref{eq:Ef1}, which will be mainly used to control the solution in the exterior region. Inspired by the method in \cite{yang:NLW:1D:p}, we make use of the Lorentz rotation in this region. 

In the energy identity \eqref{eq:energy:id}, choose the vector fields and function $\chi$ as follows $$X=x_1\partial_t+t\partial_1,\quad Y=0,\quad \chi=0.$$
 Then $ \pi^X=0$ and
\begin{align*}
&div(Y)+T[\phi]^{\mu\nu}\pi^X_{\mu\nu}+
\chi \pa_\mu\phi \pa^\mu\phi  +\chi\phi\Box\phi-\f12 \Box\chi |\phi|^2=0.
\end{align*}
Let the domain $\mathcal{D}$ be $\mathcal{J}^{-}(q)\cap\{x_1\geq t\}$ with boundary $(\B_{(0, x_0)}(t_0)\cap\mathcal{D})\cup (\mathcal{N}^{-}(q)\cap\mathcal{D})\cup(\{x_1= t\}\cap\mathcal{D})$. By using Stokes' formula, we have the weighted energy conservation adapted to these boundaries. For the integral on the initial hypersurface $\B_{(0, x_0)}(t_0)$, we have
\begin{align}
\label{eq:bxt01}
 \int_{\B_{(0, x_0)}(t_0)\cap\mathcal{D}} i_{J^{X, Y,\chi}[\phi]}d\vol 
&=\f12 \int_{\B_{(0, x_0)}(t_0)\cap\mathcal{D}} x_1(  |\partial_t\phi|^2 +|\nabla\phi|^2+\frac{2}{p+1}|\phi|^{p+1}) dx\leq C \mathcal{E}_{0, 1}.
\end{align}
On the null hypersurface $\{x_1= t\}\cap\mathcal{D}$, we have 
\[
X=x_1\pa_t+t\pa_1=t(\pa_t+\pa_1)=t L_1.
\]
Thus the surface measure is of the form
\begin{align*}
-i_{J^{X,Y, \chi}[\phi]}d\vol&=(J^{X,Y, \chi}[\phi])_{L_1}d\sigma=T[\phi]_{X L_1} d\sigma=tT[\phi]_{L_1 L_1} d\sigma=t|L_1\phi|^2d\sigma.
\end{align*}
This in particular shows that 
\begin{align}
\label{eq:t=x1}
 \int_{\{x_1= t\}\cap\mathcal{D}} -i_{J^{X, Y,\chi}[\phi]}d\vol \geq0.
\end{align}
Here keep in mind that we only consider the estimates in the future $t\geq 0$.

Finally for the integral on the backward light cone $\mathcal{N}^{-}(q)\cap\mathcal{D}$, similarly, we first can write the surface measure as
\begin{align*}
-i_{J^{X,Y, \chi}[\phi]}d\vol&=-(J^{X,Y, \chi}[\phi])_{\tilde{\Lb}}d\sigma=T[\phi]_{X \tilde{\Lb}} d\sigma.
\end{align*}
Now we need to write the vector field $X$ under the new null frame $\{\tilde{L}, \tilde{\Lb}, \tilde{e}_1, \tilde{e}_2\}$ centered at the point $q$. Note that
\begin{align*}
\pa_1=\tilde{\pa_{1}}=\tilde{\om}_1\pa_{\tilde{r}}+ \hat{e}_1\cdot(\tilde{\nabla}-\tilde{\om}\pa_{\tilde{r}}),\quad \hat{e}_1=(1,0,0).
\end{align*}
Then we have
\begin{align*}
X=x_1\partial_t+t\partial_1
&=x_1\pa_{\tilde{t}}+t(\tilde{\om}_1\pa_{\tilde{r}}+ \hat{e}_1\cdot\tilde{\nabb})=\f12 (x_1+t\tilde{\om}_1)\tilde{L}+\f12 (x_1-t\tilde{\om}_1) \tilde{\Lb}+t\hat{e}_1\cdot\tilde{\nabb}.
\end{align*}
Here $\tilde{\nabb}=\tilde{\nabla}-\tilde{\om}\pa_{\tilde{r}}$. Then we can compute the quadratic terms
\begin{align*}
T[\phi]_{X\tilde{\Lb}} 
=& \f12 (x_1-t\tilde{\om}_1) |{\tilde{\Lb}}\phi|^2 + \f12 (x_1+t\tilde{\om}_1) (|\tilde{\nabb}\phi|^2+\frac{2}{p+1}|\phi|^{p+1})+t({\tilde{\Lb}}\phi) (\hat{e}_1\cdot \tilde{\nabb})\phi.
\end{align*}
Since
\begin{align*}
  |(\hat{e}_1\cdot \tilde{\nabb})\phi|=|((\hat{e}_1-\tilde{\om}_1\tilde{\om})\cdot \tilde{\nabb})\phi|\leq |\hat{e}_1-\tilde{\om}_1\tilde{\om}||\tilde{\nabb}\phi|=\sqrt{1-\tilde{\om}_1^2}|\tilde{\nabb}\phi|,
\end{align*}
restricted to the region $\mathcal{N}^{-}(q)\cap\mathcal{D}$ where $x_1\geq t\geq 0$,  
the pure quadratic terms are nonnegative
\begin{align*}
&\f12 (x_1-t\tilde{\om}_1) |\tilde{\Lb}\phi|^2+\f12(x_1+t\tilde{\om}_1) |\tilde{\nabb}\phi|^2+t{\tilde{\Lb}}\phi(\hat{e}_1\cdot \tilde{\nabb})\phi\\
&\geq  
\f12 t \left( (1-\tilde{\om}_1) |\tilde{\Lb}\phi|^2+ (1+\tilde{\om}_1) |\tilde{\nabb}\phi|^2-2\sqrt{1-\tilde{\om}_1^2}|\tilde{\Lb}\phi | |\tilde{\nabb}\phi| \right)\geq 0.
\end{align*}
In particular on $\mathcal{N}^{-}(q)\cap\mathcal{D}$, we have
\begin{align*}
T[\phi]_{X\tilde{\Lb}} 
\geq& \frac{x_1+t\tilde{\om}_1}{p+1}|\phi|^{p+1}.
\end{align*}
This leads to the lower bound 
\begin{align}
\label{eq:N-qD}
 \int_{\mathcal{N}^{-}(q)\cap\mathcal{D}} -i_{J^{X, Y,\chi}[\phi]}d\vol \geq\int_{\mathcal{N}^{-}(q)\cap\mathcal{D}}\frac{x_1+t\tilde{\om}_1}{p+1}|\phi|^{p+1}d\sigma.
\end{align}
For such choice of vector fields, we have the weighted energy conservation
\begin{align*}
  \int_{\mathcal{N}^{-}(q)\cap\mathcal{D}}i_{J^{X, Y, \chi}[\phi]}d\vol+\int_{\{x_1= t\}\cap\mathcal{D}} i_{J^{X, Y,\chi}[\phi]}d\vol+\int_{\B_{(0, x_0)\cap\mathcal{D}}(t_0)}i_{J^{X, Y, \chi}[\phi]}d\vol= 0.
\end{align*}
In view of the above estimates  \eqref{eq:bxt01}, \eqref{eq:t=x1}, \eqref{eq:N-qD}, we conclude that
\begin{align}\label{eq:N-q1}
   &\int_{\mathcal{N}^{-}(q)\cap\mathcal{D}}\frac{x_1+t\tilde{\om}_1}{p+1}|\phi|^{p+1}d\sigma\leq C\mathcal{E}_{0, 1}.
\end{align}
Under the coordinates $(\tilde{t}, \tilde{x})$ centered at $q=(t_0, r_0,0,0)$, we have
\begin{align*}
&(t,x)=(t_0+\widetilde{t},r_0+\widetilde{x}_1,\widetilde{x}_2,\widetilde{x}_3)
=(t_0+\widetilde{t},r_0+\widetilde{r}\widetilde{\omega}_1,\widetilde{r}\widetilde{\omega}_2,\widetilde{r}\widetilde{\omega}_3).
\end{align*}
Note that $ \widetilde{t}=-\widetilde{r}$ on $\mathcal{N}^{-}(q)$. We then can write 
$$0\leq x_1+t\tilde{\om}_1=r_0+\widetilde{r}\widetilde{\omega}_1+(t_0-\widetilde{r})\tilde{\om}_1=r_0+t_0\tilde{\om}_1$$
 on $\mathcal{N}^{-}(q)\cap\mathcal{D}$. The uniform bound \eqref{eq:Ef1} then follows from \eqref{eq:N-q1} by noting that  $$ \mathcal{N}^{-}(q)\cap\mathcal{D}= \mathcal{N}^{-}(q)\cap\{x_1\geq t\}.$$

 \end{proof}

\section{Asymptotic pointwise behaviors for the solutions}
Following the framework developed in \cite{yang:NLW:ptdecay:3D}, we now use the weighted energy estimates through the backward light cone obtained in the previous section to control the nonlinearity.  
  For this purpose, we need the following integration bound: for constants $A>0,$ $B>0$, $ \gamma>1$, there holds
 \begin{align}
 \label{eq:AB}
    &\int_{\S_{(t_0-\tilde{r}, x_0)}(\tilde{r})} ((1+\tilde{\om}_1)A+B)^{-\gamma} d\tilde{\om} =2\pi\int_{-1}^1 ((1+\tilde{\om}_1)A+B)^{-\gamma}d\tilde{\om}_1\leq C A^{-1}B^{1-\gamma}
  \end{align}
 with constant $C$ depending only on $ \gamma$. 

Now we prove the main Theorem \ref{thm:main}.
 For any point $q=(t_0, x_0)$ in $\mathbb{R}^{1+3}$, recall the representation formula for linear wave equation
\begin{equation}
\label{eq:rep4phi:ex}
\begin{split}
4\pi\phi(t_0, x_0)&=\int_{\tilde{\om}}t_0  \phi_1(x_0+t_0\tilde{\om})d\tilde{\om}+\pa_{t_0}\big(\int_{\tilde{\om}}t_0  \phi_0(x_0+t_0\tilde{\om})d\tilde{\om}   \big)-\int_{\mathcal{N}^{-}(q)}|\phi|^{p-1} \phi \ \tilde{r} d\tilde{r}d\tilde{\om}.
\end{split}
\end{equation}
The first two terms are linear evolution, relying only on the initial data. Standard Sobolev embedding leads the decay estimate 
\begin{align*}
  |\int_{\tilde{\om}}t_0  \phi_1(x_0+t_0\tilde{\om})d\tilde{\om}+\pa_{t_0}\big(\int_{\tilde{\om}}t_0  \phi_0(x_0+t_0\tilde{\om})d\tilde{\om}   \big)|
  &\les (1+t_0+|x_0|)^{-1}\sqrt{\mathcal{E}_{1, 2} }.
\end{align*}
We control the nonlinear term by using the weighted energy estimates derived in Proposition \ref{prop:EF}. 
Without lose of generality (or by spatial rotation), we can assume that $x_0=(r_0,0,0)$ with $r_0=|x_0|$. Let $$u_0=|t_0-r_0|+1,\quad v_0=t_0+r_0+1.$$
Since $0<p-1\leq 1$, it holds that 
$$1+|t_0-r_0|^{p-1}\geq(1+|t_0-r_0|)^{p-1}=u_0^{p-1},\quad 1+|t_0+r_0|^{p-1}\geq(1+|t_0+r_0|)^{p-1}=v_0^{p-1}. $$ 
Therefore we have the lower bound 
\begin{align*}
  2(|t_0-r_0|^{p-1}(1-\tilde{\om}_1)+|t_0+r_0|^{p-1}(1+\tilde{\om}_1)+2)
  &\geq 2(1+\widetilde{\omega}_1)v_0^{p-1}+2(1-\widetilde{\omega}_1)u_0^{p-1}\\ 
  &\geq 
  (1+\widetilde{\omega}_1)v_0^{p-1}+2u_0^{p-1}
\end{align*}
for $\widetilde{\omega}_1\in[-1,1].$ 
In view of Proposition \ref{prop:EF}, we derive that
\begin{align}\label{eq:v01}
  &\int_{\mathcal{N}^{-}(q)}((1+\widetilde{\omega}_1)v_0^{p-1}+u_0^{p-1})|\phi|^{p+1}d\sigma\\ 
  \notag
  &\leq 2\int_{\mathcal{N}^{-}(q)}(|t_0-r_0|^{p-1}(1-\tilde{\om}_1)+|t_0+r_0|^{p-1}(1+\tilde{\om}_1)+2)|\phi|^{p+1}d\sigma\\
  \notag
  & \leq C\mathcal{E}_{0, 2}.
\end{align}
We first consider the case when $1<p<2.$ In the exterior region when $t_0\leq r_0$, note that the backward light cone  $ \mathcal{N}^{-}(q)$ entirely locates in the region $ \{x_1\geq t\}$. Moreover
\begin{align*}
  &4(1+r_0+t_0\widetilde{\omega}_1)=2(1+\widetilde{\omega}_1)v_0+2(1-\widetilde{\omega}_1)u_0\geq(1+\widetilde{\omega}_1)(v_0+u_0)+(1-\widetilde{\omega}_1)u_0
  =(1+\widetilde{\omega}_1)v_0+2u_0
\end{align*}
for $\widetilde{\omega}_1\in[-1,1].$ Then by Proposition \ref{prop:EF} as well as the standard energy estimate, we have
\begin{align*}
  &\int_{\mathcal{N}^{-}(q)}((1+\widetilde{\omega}_1)v_0+u_0)|\phi|^{p+1}d\sigma\leq 4\int_{\mathcal{N}^{-}(q)}(1+r_0+t_0\widetilde{\omega}_1)|\phi|^{p+1}d\sigma\leq C\mathcal{E}_{0, 2}.
\end{align*}
Under the coordinates centered at $q$, the surface measure $d\sigma$ can be written as $\tilde{r}^2d\tilde{r}d\tilde{\om}$. By using the integration bound \eqref{eq:AB} with $ \gamma=p$, we can estimate that
\begin{align*}
  &|\int_{\mathcal{N}^{-}(q) }|\phi|^{p-1}\phi \ \tilde{r} d\tilde{r}d\tilde{\om}|\\
  &\les \left(\int_{\mathcal{N}^{-}(q)}((1+\widetilde{\omega}_1)v_0+u_0)|\phi|^{p+1}\ \tilde{r}^{2} d\tilde{r}d\tilde{\om}\right)^{\frac{p}{p+1}}  \cdot \left(\int_{\mathcal{N}^{-}(q) }((1+\widetilde{\omega}_1)v_0+u_0)^{-p} \tilde{r}^{1-p} d\tilde{r}d\tilde{\om}\right)^{\frac{1}{p+1}}\\
  &\les (\mathcal{E}_{0, 2} )^{\frac{p}{p+1}} \left(\int_{0}^{t_0}  v_0^{-1} u_0^{1-p}\tilde{r}^{1-p} d\tilde{r} \right)^{\frac{1}{p+1}}\\
  &\les (\mathcal{E}_{0, 2} )^{\frac{p}{p+1}} \left(v_0^{-1} u_0^{1-p}t_0^{2-p}\right)^{\frac{1}{p+1}}\\
  &\les (\mathcal{E}_{0, 2} )^{\frac{p}{p+1}}(u_0v_0)^{\frac{1-p}{p+1}}.
\end{align*}
In particular, the solution $\phi$ verifies the following decay estimate in the exterior region
\begin{align*}
  |\phi(t_0, x_0)|\les v_0^{-1}(\mathcal{E}_{1, 2})^{\frac{1}{2}}+(u_0v_0)^{\frac{1-p}{p+1}}(\mathcal{E}_{0, 2})^{\frac{p}{p+1}}\les (u_0v_0)^{\frac{1-p}{p+1}}(\mathcal{E}_{1, 2}^{\frac{1}{2}}+\mathcal{E}_{0, 2}^{\frac{p}{p+1}}).
\end{align*}
In the interior region when $t_0\geq r_0$, we rely on the following improved weighted energy estimate
\begin{align}
\label{eq:v0u0}
  &\int_{\mathcal{N}^{-}(q)}((1+\widetilde{\omega}_1)(v_0^{p-1}+\tilde{r}u_0^{p-2})+u_0^{p-1})|\phi|^{p+1}d\sigma\leq C\mathcal{E}_{0, 2}.
\end{align}
In fact, from the above weighted energy estimate \eqref{eq:v01}, we conclude that
\begin{align}\label{eq:v02}
  &\int_{\mathcal{N}^{-}(q)\cap\{(1+\widetilde{\omega}_1)\tilde{r}\leq 2u_0\}}(1+\widetilde{\omega}_1)\tilde{r}u_0^{p-2}|\phi|^{p+1}d\sigma\leq 2\int_{\mathcal{N}^{-}(q)}u_0^{p-1}|\phi|^{p+1}d\sigma\leq C\mathcal{E}_{0, 2}.
\end{align}
On the other hand, for the point $(t,x)\in\mathcal{N}^{-}(q)$ such that $(1+\widetilde{\omega}_1)\tilde{r}\geq 2u_0$, note that 
\begin{align*}
&(t,x)=(t_0+\widetilde{t},r_0+\widetilde{x}_1,\widetilde{x}_2,\widetilde{x}_3)
=(t_0-\widetilde{r},r_0+\widetilde{r}\widetilde{\omega}_1,\widetilde{r}\widetilde{\omega}_2,\widetilde{r}\widetilde{\omega}_3),\quad 0\leq\widetilde{r}=-\widetilde{t}\leq t_0 
\end{align*}
In particular we have 
\[
x_1-t=r_0-t_0+(1+\widetilde{\omega}_1)\tilde{r}\geq 1-u_0+2u_0>0.
\]
This shows that 
\[
\mathcal{N}^{-}(q)\cap\{(1+\widetilde{\omega}_1)\tilde{r}\geq 2u_0\} \subset \mathcal{N}^{-}(q)\cap\{x_1\geq t\}.
\]
Moreover note that
\begin{align*}
&r_0+t_0\tilde{\om}_1=r_0-t_0+(1+\widetilde{\omega}_1)t_0=1-u_0+(1+\widetilde{\omega}_1)t_0\geq -\f12 (1+\widetilde{\omega}_1)\tilde{r}+(1+\widetilde{\omega}_1)\tilde{r}=\f12 (1+\widetilde{\omega}_1)\tilde{r}.
\end{align*}
Since $p\leq 2$ and $u_0\geq 1$, by Proposition \ref{prop:EF}, we can show that
\begin{align}\label{eq:v03}
  \int_{\mathcal{N}^{-}(q)\cap\{(1+\widetilde{\omega}_1)\tilde{r}\geq 2u_0\}}(1+\widetilde{\omega}_1)\tilde{r}u_0^{p-2}|\phi|^{p+1}d\sigma &\leq \int_{\mathcal{N}^{-}(q)\cap\{(1+\widetilde{\omega}_1)\tilde{r}\geq 2u_0\}}(1+\widetilde{\omega}_1)\tilde{r}|\phi|^{p+1}d\sigma \\ 
  \notag
  &\leq 2\int_{\mathcal{N}^{-}(q)\cap\{x_1\geq t\}}|r_0+t_0\tilde{\om}_1||\phi|^{p+1}d\sigma\\
\notag
  &\leq C\mathcal{E}_{0, 2}.
\end{align}
The improved estimate \eqref{eq:v0u0} then follows from \eqref{eq:v01}, \eqref{eq:v02} and \eqref{eq:v03}. 

Now using the integration bound \eqref{eq:AB} with $ \gamma=p$, we can show that
\begin{align*}
  &|\int_{\mathcal{N}^{-}(q) }|\phi|^{p-1}\phi \ \tilde{r} d\tilde{r}d\tilde{\om}|\\
  \les& \left(\int_{\mathcal{N}^{-}(q)}((1+\widetilde{\omega}_1)(v_0^{p-1}+\tilde{r}u_0^{p-2})+u_0^{p-1})|\phi|^{p+1}\ \tilde{r}^{2} d\tilde{r}d\tilde{\om}\right)^{\frac{p}{p+1}}  \cdot\\& \left(\int_{\mathcal{N}^{-}(q) }((1+\widetilde{\omega}_1)(v_0^{p-1}+\tilde{r}u_0^{p-2})+u_0^{p-1})^{-p} \tilde{r}^{1-p} d\tilde{r}d\tilde{\om}\right)^{\frac{1}{p+1}}\\
  \les& (\mathcal{E}_{0, 2})^{\frac{p}{p+1}} \left(\int_{0}^{t_0}  (v_0^{p-1}+\tilde{r}u_0^{p-2})^{-1} u_0^{-(p-1)^2}\tilde{r}^{1-p} d\tilde{r} \right)^{\frac{1}{p+1}}\\
  \les& (\mathcal{E}_{0, 2})^{\frac{p}{p+1}} \left(v_0^{-(p-1)^2}u_0^{(p-2)^2} u_0^{-(p-1)^2}\right)^{\frac{1}{p+1}}\\
  =& (\mathcal{E}_{0, 2})^{\frac{p}{p+1}}v_0^{-\frac{(p-1)^2}{p+1}}u_0^{\frac{3-2p}{p+1}}.
\end{align*}
Here we used the fact that for positive constants $A>0$, $B>0$, it holds that
\begin{align*}
  &\int_{0}^{+\infty}  (A+\tilde{r}B)^{-1} \tilde{r}^{1-p} d\tilde{r}=A^{-1}(A/B)^{2-p}\int_{0}^{+\infty}  (1+z)^{-1} z^{1-p} dz=C_pA^{1-p}B^{p-2}
\end{align*}
for some constant $C_p$ depending only on $p$. 

Therefore the solution $\phi$ satisfies the following estimate in the interior region
\begin{align*}
  |\phi(t_0, x_0)|\les v_0^{-1} \mathcal{E}_{1, 2}^{\frac{1}{2}}+v_0^{-\frac{(p-1)^2}{p+1}}u_0^{\frac{3-2p}{p+1}} \mathcal{E}_{0, 2}^{\frac{p}{p+1}}\les v_0^{-\frac{(p-1)^2}{p+1}}u_0^{\frac{3-2p}{p+1}}(\mathcal{E}_{1, 2}^{\frac{1}{2}}+\mathcal{E}_{0, 2}^{\frac{p}{p+1}}).
\end{align*}
By our convention, the implicit constant relies only on $p$. Recall the definition of $u_0$, $v_0$, $r_0$. We have shown the desired pointwise estimates for the solution of the main Theorem \ref{thm:main} for all $1<p<2$.

\bigskip

Finally to finish the proof for the main Theorem, it remains to discuss the end point case when $p=2$. 
Fix time $T>0$. Define
\begin{align*}
M=\sup\limits_{0\leq t\leq T,x\in\mathbb{R}^3} |\phi(t, x)|.
\end{align*}
In view of Remark \ref{remark2}, $M$ is finite for all $T>0$. Choose small constant $\ep$ such that $0<\epsilon<\frac{1}{2}$. From the weighted energy estimate \eqref{eq:v01} as  well as the integration bound \eqref{eq:AB}, similarly we can show that
\begin{align*}
  |\int_{\mathcal{N}^{-}(q) }|\phi| \phi \ \tilde{r} d\tilde{r}d\tilde{\om}|
  &\leq M^{\epsilon}\int_{\mathcal{N}^{-}(q) }|\phi|^{2-\epsilon}\ \tilde{r} d\tilde{r}d\tilde{\om}\\
  &\les M^{\epsilon}\left(\int_{\mathcal{N}^{-}(q)}((1+\widetilde{\omega}_1)v_0+u_0)|\phi|^{3}\ \tilde{r}^{2} d\tilde{r}d\tilde{\om}\right)^{\frac{2-\epsilon}{3}}  \\&
  \quad  \cdot\left(\int_{\mathcal{N}^{-}(q) }((1+\widetilde{\omega}_1)v_0+u_0)^{-\frac{2-\epsilon}{1+\epsilon}} \tilde{r}^{-\frac{1-2\epsilon}{1+\epsilon}} d\tilde{r}d\tilde{\om}\right)^{\frac{1+\epsilon}{3}}\\
  & \les M^{\epsilon}\mathcal{E}_{0, 2}^{\frac{2-\epsilon}{3}} \left(\int_{0}^{t_0}  v_0^{-1} u_0^{-\frac{1-2\epsilon}{1+\epsilon}}\tilde{r}^{-\frac{1-2\epsilon}{1+\epsilon}} d\tilde{r} \right)^{\frac{1+\epsilon}{3}}\\
  & \les M^{\epsilon} \mathcal{E}_{0, 2}^{\frac{2-\epsilon}{3}} \left(v_0^{-1} u_0^{-\frac{1-2\epsilon}{1+\epsilon}}t_0^{\frac{3\epsilon}{1+\epsilon}}\right)^{\frac{1+\epsilon}{3}}\\
 & \les M^{\epsilon} \mathcal{E}_{0, 2}^{\frac{2-\epsilon}{3}} (u_0v_0)^{-\frac{1-2\epsilon}{3}}.
\end{align*}
Since  $0< \epsilon<1/2$, this shows that
\begin{align*}
  |\phi(t_0, x_0)|\les v_0^{-1} \mathcal{E}_{1, 2}^{\frac{1}{2}}+M^{\epsilon} \mathcal{E}_{0, 2}^{\frac{2-\epsilon}{3}} (u_0v_0)^{-\frac{1-2\epsilon}{3}}\les \mathcal{E}_{1, 2}^{\frac{1}{2}}+M^{\epsilon}\mathcal{E}_{0, 2}^{\frac{2-\epsilon}{3}}.
\end{align*}
Hence taking supreme in terms of $t_0,$ $x_0$ and in view of the definition for $M$, we derive that
\begin{align*}
M\les \mathcal{E}_{1, 2}^{\frac{1}{2}}+M^{\epsilon}\mathcal{E}_{0, 2}^{\frac{2-\epsilon}{3}},
\end{align*}
from which we conclude that 
\begin{align*}
M\les \mathcal{E}_{1, 2}^{\frac{1}{2}}+\mathcal{E}_{0, 2}^{\frac{2-\epsilon}{3(1-\epsilon)}}.
\end{align*}
Here we note that $p=2$. This leads to the pointwise estimate for the solution for the case $p=2$
\begin{align*}
|\phi(t_0, x_0)|\les v_0^{-1}\mathcal{E}_{1, 2}^{\frac{1}{2}}+M^{\epsilon}\mathcal{E}_{0, 2}^{\frac{2-\epsilon}{3}} (u_0v_0)^{-\frac{1-2\epsilon}{3}}\les(u_0v_0)^{-\frac{1-2\epsilon}{3}}(\mathcal{E}_{1, 2}^{\frac{1}{2}}+\mathcal{E}_{0, 2}^{\frac{2-\epsilon}{3(1-\epsilon)}}).
\end{align*}
From the proof, we see that the implicit constant relies only on $p$ and $\epsilon$. We thus finished the proof for the main Theorem \ref{thm:main}.

\bibliography{shiwu}{}

\begin{thebibliography}{10}

\bibitem{Baez:3DNLW:Groursat}
J.~Baez, I.~Segal, and Z.~Zhou.
\newblock The global {G}oursat problem and scattering for nonlinear wave
  equations.
\newblock {\em J. Funct. Anal.}, 93(2):239--269, 1990.

\bibitem{Bieli:3DNLW}
R.~Bieli and N.~Szpak.
\newblock {Large data pointwise decay for defocusing semilinear wave
  equations}.
\newblock 2010.
\newblock ar{X}iv:1002.3623.

\bibitem{newapp}
M.~Dafermos and I.~Rodnianski.
\newblock A new physical-space approach to decay for the wave equation with
  applications to black hole spacetimes.
\newblock In {\em X{VI}th {I}nternational {C}ongress on {M}athematical
  {P}hysics}, pages 421--432. World Sci. Publ., Hackensack, NJ, 2010.

\bibitem{Igor20:blow:NLS}
F.Merle, P.~Raphael, I.~Rodnianski, and J.~Szeftel.
\newblock {On blow up for the energy super critical defocusing non linear
  Schr\"{o}dinger equations }.
\newblock 2019.
\newblock ar{X}iv:1912.11005.

\bibitem{velo85:global:sol:NLW}
J.~Ginibre and G.~Velo.
\newblock The global {C}auchy problem for the nonlinear {K}lein-{G}ordon
  equation.
\newblock {\em Math. Z.}, 189(4):487--505, 1985.

\bibitem{Velo87:decay:NLW}
J.~Ginibre and G.~Velo.
\newblock Conformal invariance and time decay for nonlinear wave equations.
  {I}, {II}.
\newblock {\em Ann. Inst. H. Poincar\'{e} Phys. Th\'{e}or.}, 47(3):221--261,
  263--276, 1987.

\bibitem{Pecher82:NLW:2d}
R.~Glassey and H.~Pecher.
\newblock Time decay for nonlinear wave equations in two space dimensions.
\newblock {\em Manuscripta Math.}, 38(3):387--400, 1982.

\bibitem{Grillakis:NLW:cri:3d}
M.~Grillakis.
\newblock Regularity and asymptotic behaviour of the wave equation with a
  critical nonlinearity.
\newblock {\em Ann. of Math. (2)}, 132(3):485--509, 1990.

\bibitem{Hidano:scattering:NLW}
K.~Hidano.
\newblock Scattering problem for the nonlinear wave equation in the finite
  energy and conformal charge space.
\newblock {\em J. Funct. Anal.}, 187(2):274--307, 2001.

\bibitem{Kapitanski94:NLW:n9:cri}
L.~Kapitanski.
\newblock Global and unique weak solutions of nonlinear wave equations.
\newblock {\em Math. Res. Lett.}, 1(2):211--223, 1994.

\bibitem{tao12:1d:NLW}
H.~Lindblad and T.~Tao.
\newblock Asymptotic decay for a one-dimensional nonlinear wave equation.
\newblock {\em Anal. PDE}, 5(2):411--422, 2012.

\bibitem{mora1}
C.~S. Morawetz.
\newblock The limiting amplitude principle.
\newblock {\em Comm. Pure Appl. Math.}, 15:349--361, 1962.

\bibitem{Pecher82:decay:3d}
H.~Pecher.
\newblock Decay of solutions of nonlinear wave equations in three space
  dimensions.
\newblock {\em J. Funct. Anal.}, 46(2):221--229, 1982.

\bibitem{Strauss:NLW:decay}
W.~Strauss.
\newblock Decay and asymptotics for {$ cmu=F(u)$}.
\newblock {\em J. Functional Analysis}, 2:409--457, 1968.

\bibitem{yang:NLW:2D}
D.~Wei and S.~Yang.
\newblock {On the global behaviors for defocusing semilinear wave equations in
  $\mathbb{R}^{1+2}$}.
\newblock {\em {Analysis \& PDE}}.
\newblock to appear.

\bibitem{yang:NLW:1D:p}
D.~Wei and S.~Yang.
\newblock Asymptotic decay for defocusing semilinear wave equations in
  $\mathbb{R}^{1+1}$.
\newblock {\em Ann. PDE}, 7(9):3, 2021.

\bibitem{yang:scattering:NLW}
S.~Yang.
\newblock {Global behaviors for defocusing semilinear wave equations}.
\newblock {\em Ann. Sci. \'{E}c. Norm. Sup\'{e}r. (4)}.
\newblock to appear.

\bibitem{yang:NLW:ptdecay:3D}
S.~Yang.
\newblock {Pointwise decay for semilinear wave equations in
  $\mathbb{R}^{1+3}$}.
\newblock 2019.
\newblock ar{X}iv:1908.00607.

\end{thebibliography}
\bibliographystyle{plain}

\bigskip
School of Mathematical Sciences, Peking University, Beijing, China

\textsl{Email}: {jnwdyi@pku.edu.cn}

Beijing International Center for Mathematical Research, Peking University,
Beijing, China

\textsl{Email}: shiwuyang@math.pku.edu.cn

\end{document}